\documentclass[12pt]{amsart}
\usepackage{amsmath, amsthm, amsfonts, amssymb}

\newtheorem{theorem}{Theorem}[section]

\newtheorem{corollary}[theorem]{Corollary}
\theoremstyle{definition}
\newtheorem{definition}[theorem]{Definition}
\newtheorem{example}[theorem]{Example}

\theoremstyle{remark}
\newtheorem{remark}[theorem]{Remark}
\numberwithin{equation}{section}
\newtheorem{question}[theorem]{Question}
\newtheorem{counterexample}[theorem]{Counterexample}

\renewcommand{\[}{\left[}

\newcommand{\mC}{\ensuremath{\mathbb{C}}}
\newcommand{\mD}{\ensuremath{\mathbb{D}}}

\newcommand{\mN}{\ensuremath{\mathbb{N}}}

 \allowdisplaybreaks
 
\begin{document}

\title{Some Normality Criteria and a Counterexample to the Converse of Bloch's Principle}

\author[K. S. Charak]{Kuldeep Singh Charak}
\address{
\begin{tabular}{lll}
&Kuldeep Singh Charak\\
&Department of Mathematics\\
&University of Jammu\\
&Jammu-180 006\\ 
&India\\
\end{tabular}}
\email{kscharak7@rediffmail.com}

\author[S. Sharma]{Shittal Sharma$^*$}
\address{
\begin{tabular}{lll}
&Shittal Sharma\\
&Department of Mathematics\\
&University of Jammu\\
&Jammu-180 006\\
&India
\end{tabular}}
\email{shittalsharma{\_}mat07@rediffmail.com}

\begin{abstract}
In this paper we continue our earlier investigations on normal
families of meromorphic functions\cite{CS2}. Here, we  prove some
value distribution results which lead to some  normality criteria
for a family of meromorphic functions involving the sharing of a
holomorphic function by more general differential polynomials
generated by members of the family and get some recently known
results extended and improved. In particular, the main result 
of this paper leads to a counterexample to the converse of 
Bloch's principle.
\end{abstract}

\let\thefootnote\relax
\let\thefootnote\relax\footnote{
The work of the second author is supported by University Grants Commission(UGC), INDIA (No.F.17-77/08(SA-1)) .}
\let\thefootnote\relax\footnote{*: Corresponding author}

\maketitle

\section{Introduction and Main Results}
A family $\mathcal{F}$ of meromorphic functions in a complex
domain $D$ is said to be {\it normal} in $D$ if every sequence in
$\mathcal{F}$ has a subsequence that converges uniformly on
compact subsets of $D$ with respect to the spherical metric. 
The concept of normality was introduced in
1907 by P. Montel \cite{PM}. Though normal families play a central
role in complex dynamics, yet it is a subject of great interest in
its own right. For normal families of meromorphic functions, the
reader may refer to Joel Schiff's book \cite{JS}, Zalcman's survey
article \cite{LZ}, Drasin's paper \cite{DD} out of a huge
literature on the subject. {\it To find out normality criteria} is
a common research problem in the theory of normal families. It is
David Drasin \cite{DD} who brought Nevanlinna value distribution
theory \cite{WK} in the study of normality of families of
meromorphic functions and Wilhelm Schwick \cite{SCH} introduced
the concept of sharing of values in the study of normal families.
In this paper we  prove a value distribution result leading to some 
interesting normality criteria one of which leads to a construction 
of a counterexample to the converse of the Bloch's principle. These 
normality criteria in fact involve the sharing of holomorphic functions
by a more general class of differential polynomials and get some 
recently known results generalized and improved. This work, in fact, is in 
continuation to our earlier work \cite{CS2}.\\

\medskip

Let $f \in \mathcal {F}$ and  $h(z)$ be a holomorphic function on
$D$. Let $k, l_0, l_1,l_2, \cdots, l_k$, $ m_1, m_2, \cdots, m_k$
be non-negative integers with $l' = \sum^{k}_{i=1}{l_i}$ and
 $ m' = \sum^{k}_{i=1}{m_i} $ and let
 $$P[f]= f^{l_0}(f^{l_1})^{(m_1)}(f^{l_2})^{(m_2)} \cdots (f^{l_k})^{(m_k)}, (k \geq 1)$$
  be a differential polynomial of $f \in \mathcal{F}$ with degree $\gamma _{P} = l_0 + l',$ where
  $ l_0 > 0 \text{ and } l_i \geq m_i, \text{ for all } i: 1\leq i \leq k$ with $l'>m'>0.$\\
  Further, we can see that
  $$(f^{l_i})^{(m_i)} = \sum{C_{n_0 n_1 n_2 \cdots n_{m_i}} f^{n_0} (f')^{n_1} (f'')^{n_2} \cdots (f^{(m_i)})^{n_{m_i}}}$$
  is such that $\sum^{m_i}_{j=0}{n_j} = l_i \text{ and } \sum^{m_i}_{j=1}{jn_j} = m_i$.
  Thus, weight $$w((f^{l_i})^{(m_i)}) = \max\left\{\sum^{m_i}_{j=0}{(j+1)n_j}\right\} = \max (m_i + l_i) = l_i +m_i $$
  and so $$w(P[f])= l_0 + \sum^{k}_{i=1}{(l_i + m_i)} = l_0 + l' + m' = \gamma_{P} + m'.$$
  It is assumed that the reader is familiar with the standard notions used in the Nevanlinna value distribution theory such as
  $ m(r,f), N(r,f), T(r,f), S(r,f),$ etc.(see \cite{WK}).

 \medskip

 \begin{definition}
 We say that the two meromorphic functions $f$ and $g$ in a domain $D$ {\bf share} the function $h$ IM in $D$ if
 $\overline{E}(h,f)= \overline{E}(h,g),$ where $\overline{E}(h,\phi) =  \{ z \in D: \phi(z)-h(z)=o \}$, the set of
 zeros of $\phi - h$ in $D$ counted with ignoring multiplicities. If $\overline{E}(h,f) \subseteq \overline{E}(h,g),$
 then we say that $f$  {\bf shares} $h$ {\bf partially}
 with $g$ on $D.$
 \end{definition}

 G. Dethloff, T.V. Tan and N.V. Thin (\cite{DTT}, Corollary 2, p- 676 ) proved the
 following Picard type theorem:

 \medskip

 {\bf {Theorem A.}} \  \textit{ Let $a$ be a non-zero complex value, $l_0$ be a non-negative integer,
 and $l_1, l_2, \cdots, l_k$, $ m_1, m_2, \cdots, m_k$ be positive integers. Let $\mathcal{F}$ be a family of meromorphic
 functions in a complex domain $D$ such that for any $f \in \mathcal{F}$, $P[f]-a$ is no-where vanishing on $D$. Assume that \\
 (a) $l_j \geq m_j$; $ \forall  j:1\leq j \leq k  $\\
 (b) $l_0 +l' \geq 3+m'.$\\
 Then $\mathcal{F}$ is normal in $D$. }\\

 \medskip

 By replacing the condition, "$P[f]-a$ is no-where vanishing on $D$" with the condition "$P[f]$ and $P[g]$ share $a$ IM on $D$
 for every pair $f,g \in \mathcal{F}$" in Theorem A, G.Dutt and S.Kumar (\cite{DK}, Theorem 1.4, p- 2) obtained the
 following result:

 \medskip

{\bf {Theorem B.}} \  \textit{ Let $(0\neq) a \in \mC,$ $l_0$
be a non-negative integer and $l_1, l_2, \cdots, l_k$,
$m_1, m_2, \cdots,m_k$ be positive integers such that \\
 (a) $l_j \geq m_j$; for all $j: 1 \leq j \leq k $\\
 (b) $l_0 +l' \geq 3+m'.$\\
 Let $\mathcal{F} $ be a family of meromorphic functions in a domain $D$ such that for every pair $f,g \in \mathcal{F}$,
 $P[f]$ and $P[g]$ share $a$ IM on $D$. Then $\mathcal{F}$ is normal in $D.$ }\\

\medskip

It is now natural to consider the following more general question:
\begin{question}
Is the family $\mathcal {F}$ normal in $D$ if for each pair of
functions $f$ and $g$ in $\mathcal{F}$ the corresponding
differential polynomials $P[f]$ and $P[g]$ share a holomorphic
function $h$ IM?
\label {Question}
\end{question}

Here, in this paper we answer Question \ref{Question} as follows:

\begin{theorem}
  Let $\mathcal{F}$ be a family of non-constant meromorphic functions on a domain $D$ such that each $f \in \mathcal{F}$ has
   poles, if any, of multiplicity at least $l_0$. Let $h \not\equiv 0$ be a holomorphic function on
  $D$ having only zeros of multiplicity at most $ l_0 -1.$  If $P[f]$ and $P[g]$ share $h$ IM on $D$ for each pair
  $f,g \in \mathcal{F}$, then $\mathcal{F}$ is normal in $D.$
\label{THM4}
\end{theorem}

\begin{example}
Let $D=\mD,$ the open unit disk. Consider the family 
$$\mathcal{F} = \{f_n : f_n (z)= e^{nz^2}, z \in \mD \}$$
of meromorphic functions on $\mD.$ Then $f_n ^2 = e^{2nz^2},$ and $f_n ' =2nze^{nz^2}.$
Let $P[f] =f(f^2)' = 2f^2 f'.$ Then $P[f_n](z)=2f_n ^2 (z) f_n '(z) = 4nze^{3nz^2}.$\\
Therefore, for distinct $m,n$, we see that $P[f_m] \text{ and } P[f_n]$ share $h \equiv 0,$ IM. But the family $\mathcal{F}$ fails to be normal at $z=0$ in $\mD,$ since $f_n(0)=1; \forall n$ and $f_n(z)\longrightarrow \infty,$ for all $z\neq 0 $ in $\mD.$
\label{ex1}
\end{example}
 Example \ref{ex1} shows that the condition $h\not\equiv 0$ in Theorem \ref{THM4} is essential.\\

A direct consequence of Theorem \ref {THM4} is the following important result:

\begin{corollary}
Let $\mathcal{F}$ be a family of non-constant meromorphic functions on a domain $D$. Let $h \not\equiv 0$ be a holomorphic function such that $h(z)\neq 0$ in $D$. If $P[f]-h$ has no zero in $D; \forall f\in \mathcal{F}$, then $\mathcal{F}$ is normal in $D$.
\label{COR}
\end{corollary}
Corollary \ref{COR} is important in the sense that it leads to the construction of a counterexample to the converse of the Bloch's principle.\\

The Bloch's principle(see \cite{WB}) states that a family of holomorphic (meromorphic) functions satisfying a property $\textbf{P}$ in a domain $D$ is likely to be  normal if the property $\textbf{P}$ reduces every holomorphic (meromorphic) function on $\mC$ to a constant. The Bloch's principle is not universally true, for example one can see \cite{LAR}.\\

The converse of the Bloch's principle states that if a family of 
meromorphic functions satisfying a property $\textbf{P}$ on an 
arbitrary domain $D$ is  normal, then every meromorphic function 
on $\mC$ with property $\textbf{P}$ reduces to a constant. Like 
Bloch's principle, its converse is not true. For counterexamples 
one can see \cite{CJ} \cite{CV}, \cite{IL}, \cite{BQL}, \cite{JS}, 
\cite{XC},  and \cite{YZ}.

\begin{counterexample}
Let $P[f]= f(f^3)''= f(3f^2 f')'= 3f^3 f'' + 6 f^2 f'^2$\\
and let $f(z)=e^{-z}$ be defined on $\mC$. Then 
$$P[f](z)= 3e^{-3z} e^{-z} + 6e^{-2z}e^{-2z} = 9e^{-4z}$$
Take $h(z)= e^{-4z},$ such that $h \not\equiv 0$ and $h$ is 
holomorphic in $\mC$ and hence in every domain $D\subseteq \mC,$ 
and also $h(z)\neq 0, \forall z \in D.$ Then $(P[f]-h)(z)= 8e^{-4z}$ 
has no zeros in $\mC.$ \\
Note that $f$ is non-constant, which violates the statement of the 
converse of the Bloch's Principle in view of Corollary \ref{COR}.  
\end{counterexample} 
 Next we discuss normality of $\mathcal {F}$ when  $P[f]-h$ has zeros 
 under different situations as follows:

 \begin{theorem}
Let $\mathcal{F}$ be a family of non-constant meromorphic
functions on a domain $D$. Let $h$ be a holomorphic function on
$D$ such that $h(z)\neq 0$ in $D$. If, for each $f \in
\mathcal{F}$, any one of the following three conditions holds:\\
(i)~~~~ $(P[f]-h)(z)$ has at most one zero,\\
(ii)~~~ $(P[f]-h)(z)=0$ implies $|f(z)|\geq M$, for some $M>0,$\\
(iii)~~ $(P[f]-h)(z)=0$ implies $|(f^{l_i})^{(m_i)} (z)| \leq M,$
  for some positive $M, l_i \text{ and } m_i;$\\
  then $\mathcal{F}$ is normal in $D.$
\label{THM5}
\end{theorem}

 Further, under the weaker hypothesis of partial sharing (see \cite{CS1}, \cite{CS2}) of holomorphic functions,
 we can prove the following result:

 \begin{theorem}
 Let $\mathcal{F}$ be a family of non-constant meromorphic functions on a domain $D$. Let $h$ be a holomorphic function on $D$
 such that $h(z)\neq 0$ in $D$. If, for every $f \in \mathcal{F}$, there exists $\tilde{f} \in \mathcal{F}$ such that $P[f]$
 shares $h$ partially with $P[\tilde{f}]$, then $\mathcal{F}$ is normal in $D$, provided $h \not\equiv P[\tilde{f}]$ in $D$.
 \label{THM7}
 \end{theorem}
 \begin{remark}
 Theorem \ref{THM4}  improves and generalizes Theorem A and Theorem B. Theorem \ref{THM7} is a
 direct generalization of  Theorem 1.3 in \cite{CS2}.
 \end{remark}

\section{Some Value Distribution Results}

To facilitate the proofs of our main theorems, we prove some value
distribution results.

 \begin{theorem}
 Let $f$ be a transcendental meromorphic function. Then $P[f](z)- \omega(z)$ has infinitely many zeros, for any small function
 $\omega (\not\equiv 0, \infty)$ of $f.$
\label{THM1}
 \end{theorem}

\begin{proof} Suppose on the contrary that $P[f](z) - \omega (z)$
has only finitely many zeros. Then by Second Fundamental theorem
of Nevanlinna for three small functions(\cite{WK}, Theorem 2.5, p-47),
we find that
 $$[1+o(1)]T(r,P)\leq \overline{N}(r,P)+\overline{N}\left(r,\frac{1}{P}\right)+\overline{N} \left(r,\frac{1}{P-\omega}\right)+S(r,P)$$
   \begin{equation}
    = \overline{N}(r,P)+ \overline{N}\left(r,\frac{1}{P} \right)+ S(r,P) \label{eq2.1}
  \end{equation}
  Since  $P[f]$ is a homogeneous differential polynomial with each monomial having positive exponents of $f$, by \cite{AP}(Theorem 1, p-792),
  $f$ and $P[f]$ have the same order of growth and hence $T(r,\omega)= S(r,P) \text{ as } r\rightarrow \infty$.
  That is, $\omega$ is a small function of $f$ iff $\omega$ is a small function of $P[f].$\\
 Next,
 \begin{align*}
 \overline{N} \left( r,\frac{1}{P} \right) & = \overline{N} \left(r, \frac{1}{f^{l_0} (f^{l_1})^{(m_1)}\cdots (f^{l_k})^{(m_k)}}
 \right)\\[0.25\baselineskip]
 & \leq \overline{N} \left(r,\frac{1}{f} \right)+ \sum^{k}_{i=1}\overline{N}_0  \left(r,\frac{1}{(f^{l_i})^{(m_i)}}
 \right)\\[0.25\baselineskip]
 & \leq \overline{N} \left(r,\frac{1}{f} \right)+ \sum^{k}_{i=1} N_0  \left(r,\frac{1}{(f^{l_i})^{(m_i)}} \right),
 \end{align*}
 where $N_0 (r,\frac{1}{(f^{l_i})^{(m_i)}})$ represents the count of those zeros of $(f^{l_i})^{(m_i)}$ which are not
 the zeros of $f^{l_i} $ and hence not of $f$. Denoting by $\overline{N}_{p)}(r, \frac{1}{f})$ and
 $\overline{N}_{(p+1}(r, \frac{1}{f})$, the counting
 functions ignoring multiplicities of those zeros of $f$ whose multiplicity is at most $p$ and at least $p+1$ respectively.
 Therefore,

 \begin{align*}
 \overline{N}\left(r,\frac{1}{P}\right) & \leq
\overline{N}\left(r,\frac{1}{f} \right) + \sum^{k}_{i=1} \left[m_i
\overline{N}(r,f)+N_{m_i)} \left(r,\frac{1}{f}  \right)+ m_i
\overline{N}_{(m_i +1} \left(r,\frac{1}{f} \right) \right] +
S(r,f)\\[0.25\baselineskip]
 & \leq \overline{N} \left(r,\frac{1}{f} \right)+ \sum^{k}_{i=1}m_i \left[ \overline{N}(r,f)+\overline{N}_{m_i)} \left(r,\frac{1}{f} \right)+
  \overline{N}_{(m_i +1} \left(r,\frac{1}{f} \right) \right] + S(r,f)\\[0.25\baselineskip]
	& =  \overline{N} \left(r,\frac{1}{f} \right)+\sum^{k}_{i=1}m_i \left[ \overline{N}(r,f)+ \overline{N} \left(r,\frac{1}{f} \right) \right] +S(r,f)\\[0.25\baselineskip]
 & = \overline{N} \left(r,\frac{1}{f} \right)+ m' \left[ \overline{N}(r,f)+\overline{N} \left(r,\frac{1}{f} \right) \right]+
 S(r,f).
\end{align*}

That is,
\begin{equation}
\overline{N}\left(r,\frac{1}{P}\right) \leq
m'\overline{N}(r,f)+(1+m') \overline{N}\left(r,\frac{1}{f}
\right)+S(r,f) \label{eq2.2}
\end{equation}
Next, if $z_0$ is a zero of $f$ of order $p: 2 \leq p \leq k,$
then $z_0$ is a zero of $P[f]$ of order  $= pl_0+pl' - m' \geq
2l_0+2l' - m' \geq 2l_0+m' \geq 2+m' $. Similarly, for $p\geq
k+1,$ $z_0$ is a zero $P[f]$ of order $\geq (k+1) (l_0 + l')- m' \geq
(k+1)+km'=k(1+m') + 1 .$ Thus, we see that
$$N \left(r,\frac{1}{P} \right)-\overline{N} \left(r,\frac{1}{P} \right) \geq (m'+1)\overline{N}_{k)}\left(r,\frac{1}{f}\right)
+k(m'+1)\overline{N}_{( k+1 }\left(r,\frac{1}{f}\right)$$
That is,
$$\overline{N}_{ k)}\left(r,\frac{1}{f}\right) \leq
\dfrac{1}{m' + 1} \left[N \left(r,\frac{1}{P}\right)-\overline{N} \left(r,\frac{1}{P}\right) \right]-
k\overline{N}_{( k+1}\left(r,\frac{1}{f}\right).$$

Since $(1-k)(1+m')\leq 0$, for $k\geq 1,$  (\ref{eq2.2}) with the
help of the last inequality gives

\begin{align*}
\overline{N} \left(r,\frac{1}{P} \right)  & \leq m' \overline{N}(r,f)+ (1+m')\overline{N}_{k)} \left(r,\frac{1}{f}\right) +
 (1+m')\overline{N}_{(k+1}\left(r,\frac{1}{f} \right)+ S(r,f)\\[0.25\baselineskip]
& \leq m'\overline{N}(r,f)+ N\left(r,\frac{1}{P}\right) -\overline{N}\left(r,\frac{1}{P}\right)+
(1-k)(1+m')\overline{N}_{(k+1}\left(r,\frac{1}{f} \right)\\
& + S(r,f)\\[0.25\baselineskip]
& \leq m'\overline{N}(r,f)+ N\left(r,\frac{1}{P}\right)
-\overline{N}\left(r,\frac{1}{P}\right) +S(r,f),
 \end{align*}

\begin{equation}
\Rightarrow \overline{N}\left(r,\frac{1}{P} \right)\leq
\dfrac{m'}{2} \overline{N}(r,f)+
\dfrac{1}{2}N\left(r,\frac{1}{P}\right)+S(r,f).\label{eq3}
\end{equation}
Putting (\ref{eq3}) into (\ref{eq2.1}) and noting that
$\overline{N}(r,f)= \overline{N}(r,P) \text{ and } S(r,f)=
S(r,P)$, we get
\begin{equation}
[1+o(1)]T(r,P)\leq \left[1+\dfrac{m'}{2}
\right]\overline{N}(r,f)+\dfrac{1}{2}N\left(r,\frac{1}{P}\right)+S(r,P).\label{eq2.3}
\end{equation}
Also, a pole of $f$ of order $p \geq 1,$ is a pole of $P[f]$ of
order $pl_0 + pl' + m' \geq l_0 + l' +m' \geq 1+m'+1+m'=2+2m' $
and therefore,
$$N(r,P) \geq ( 2+2m') \overline{N}(r,f)$$
$$\Rightarrow \overline{N}(r,f) \leq \dfrac{1}{2+2m'} N(r,P)$$

and hence, (\ref{eq2.3}) yields,
 $$[1+o(1)]T(r,P)\leq \left[\dfrac{2+m'}{4(1+m')} \right]N(r,P)+\dfrac{1}{2}N\left(r,\frac{1}{P}\right) + S(r,P)$$
 $$\Rightarrow \left[1-\dfrac{1}{2}-\dfrac{2+m'}{4(1+m')}+o(1) \right]T(r,P)\leq S(r,P)$$
 $$\Rightarrow  \left[\dfrac{ m'}{4(1+m')}+o(1) \right]T(r,P)\leq S(r,P)$$
 $\Rightarrow T(r,P) \leq S(r,P,),$ which is a contradiction.\\
 Hence the result follows.
\end{proof}

 \begin{theorem}
  Let $\omega (z) \not\equiv 0$ be a polynomial of degree $m < l_0.$ Let $f$ be a non-constant rational function having poles, if any, 
  of multiplicity at least $l_0$. Then $P[f]- \omega$ has at least two distinct zeros.
  \label{THM3}
 \end{theorem}
 Note that for $m=0$, Theorem \ref {THM3} holds without any restriction on the multiplicity of poles of $f.$
 
 \medskip

 Though the proof of Theorem \ref{THM3} is based on the ideas
 from \cite{CS2} but the levels of modifications and computations
 are little involved and so we present a complete proof here.

 \medskip

\begin{proof} Suppose on the contrary that $P[f]- \omega$ has at
most one zero. We consider the following cases:

  \medskip

  {\bf Case-1:} If $f$ is a non-constant polynomial, then $P[f]$ is also a polynomial of degree at least
  $l_0 + l' -m' \geq l_0 +1.$ Since $\omega (z)$ is a polynomial of degree $m < l_0$,  $P[f](z)- \omega (z)$
  is a polynomial of degree $\geq 1.$ By Fundamental theorem of Algebra, $P[f]- \omega$ has exactly one zero.
  We can set
  \begin{equation}
  P[f](z)- \omega(z) = A (z-z_0)^n ,\label{eq2.11}
  \end{equation}
    where $A$ is a non-zero constant and $n> m+1.$ Then
    $$\frac{d^{m+1} P[f]}{dz^{m+1}}(z)= P^{(m+1)} [f](z)= A n (n-1)(n-2)\cdots (n-m) (z-z_0)^{n-m-1}$$
    which implies that $z_0$ is the only zero of $P^{(m+1)}
    [f](z)$. Since each zero of $f$ is a zero of $P[f]$ of order at least $l_o+l'-m' > m+1,$ it follows that $z_0$ is a
    zero of $P[f]$ also. Thus $P^{(m)}[f](z_0)=0$. But (\ref{eq2.11}) gives that $P^{(m)}[f](z_0)= \omega ^{(m)}(z_0) \neq 0,$
    which is a contradiction.

    \medskip

    {\bf Case-2:} When $f$ is a rational function but not a polynomial. Consider
    \begin{equation}
    f(z) = A \dfrac{\prod^{s}_{j=1}(z-\alpha_j)^{n_j}}{\prod^{t}_{j=1}(z-\beta_j)^{p_j}}, \label{eq2.12}
    \end{equation}
    where $A$ is a non-zero constant with $n_j \geq 1(j=1,2, \cdots, s)$ and $p_j \geq l_0 (j=1,2,\cdots, t).$\\

    Put
    \begin{equation}
    \sum^{s}_{j=1}n_j= S \text{ and } \sum^{t}_{j=1}p_j = T.\label{eq2.13}
    \end{equation}
    Thus $S \geq s$ and $T\geq l_0t \geq t$.

    \medskip

    We see from (\ref{eq2.12}) that
    \begin{equation}
    P = P[f](z) = \dfrac{\prod^{s}_{j=1}{(z-\alpha_j)^{n_j(l_0 + l') -m'}}}{\prod^{t}_{j=1}{(z-\beta_j)^{p_j(l_0 + l') +m'}}}
    g_{P}(z) = \frac{p(z)}{q(z)}, \label{eq2.14} \text{ say},
    \end{equation}
   where $ g_{P}(z)$ is a polynomial of degree at most $m'(s+t-1)$.\\

   On differentiating (\ref{eq2.14}), we have
   \begin{equation}
   P^{(m)} = \dfrac{\prod^{s}_{j=1}{(z-\alpha_j)^{n_j(l_0 +l') -(m'+m)}}}{\prod^{t}_{j=1}{(z-\beta_j)^{p_j(l_0 + l') + (m'+m)}}}
    \tilde{g}(z), \label{eq2.15}
   \end{equation}
   where $\tilde{g}$ is a polynomial such that $\deg (\tilde{g}) \leq (m'+m)(s+t-1).$\\

   And
   \begin{equation}
   P^{(m+1)} = \dfrac{\prod^{s}_{j=1}{(z-\alpha_j)^{n_j(l_0 +l') -(m'+m+1)}}}{\prod^{t}_{j=1}{(z-\beta_j)^{p_j(l_0 + l') +
   (m'+m+1)}}} \tilde{\tilde{g}}(z), \label{eq2.16}
   \end{equation}
    where $\tilde{\tilde{g}}$ is a polynomial of degree at most $(m'+m+1)(s+t-1).$

    \medskip

    {\bf Case-2.1:} We first assume that $P[f]- \omega$ has exactly one zero, say $z_0$. Thus, in view of (\ref{eq2.14})
    we can see that
    \begin{equation}
    P[f](z) = \omega (z) + \dfrac{B(z-z_0)^l}{\prod^{t}_{j=1}{(z-\beta_j)^{p_j(l_0 + l') + m'}}}, \label{eq2.17}
    \end{equation}
    where $l$ is a positive integer and $B$ is a non-zero constant.\\

    On differentiating (\ref{eq2.17}), we get
    \begin{equation}
    P^{(m)} = C + \dfrac{(z-z_0)^{l-m} \hat{g}(z)}{\prod^{t}_{j=1}{(z-\beta_j)^{p_j(l_0 + l') + (m'+m)}}},\label{eq2.18}
    \end{equation}
    where $ \hat{g}$ is a polynomial with degree at most $mt$ and $C \neq 0$ is a constant.\\

    And

    \begin{equation}
    P^{(m+1)} = \dfrac{(z-z_0)^{l-(m+1)} \hat{\hat{g}}(z)}{\prod^{t}_{j=1}{(z-\beta_j)^{p_j(l_0 + l') + (m'+m+1)}}},
    \label{eq2.19}
    \end{equation}
    with $\deg(\hat{\hat{g}}) \leq (m+1)t\leq l_0t.$\\

    On comparing (\ref{eq2.15}) and (\ref{eq2.18}), we see that $z_0 \neq \alpha_j (j = 1,2, \cdots, s)$ (because if
    it is not so, for some $j$, then from (\ref{eq2.15}), $z_0$ is a zero of $P^{(m)}[f]$ and from
    (\ref{eq2.18}), $P^{(m)}[f](z_0) = 0 \Rightarrow C=0 $, which is a contradiction).

    \medskip

    {\bf Case-2.1.1:} Suppose $l\neq T(l_0 + l') + tm' + m.$ Then from (\ref{eq2.17}) and using (\ref{eq2.14}),
    we find that $\deg(p)\geq \deg(q)$ and this implies that
    $$ T(l_0 +l')+tm' \leq S(l_0 + l')-m's + \deg (g_{P }) $$
    $$\Rightarrow T(l_0 +l') \leq S(l_0 + l')-m' <  S(l_0 + l') $$
    $$\Rightarrow T<S.$$
     Also, from (\ref{eq2.16}) and (\ref{eq2.19}), we see that
     $$S(l_0 + l')-(m'+m+1)s \leq \deg(\hat{\hat{g}}) \leq l_0t \leq T.$$
     \begin{align*}
     \Rightarrow S(l_0 + l') & \leq (m'+m+1)s + T\\
     & \leq (m' +l_0)S+T\\
     & < (m'+1+l_0)S\\
     & \leq (l'+l_0)S
     \end{align*}
     $\Rightarrow S<S,$ which is absurd.

     \medskip

     {\bf Case-2.1.2:} Suppose $l=T(l_0 + l')+tm'+m.$ Then, we have two possibilities: either $S>T$ or $S\leq T.$
     For the case $S>T,$ we move exactly as in the Case-2.1.1 . Therefore, we only consider the case $S\leq T$.\\

    Since (\ref{eq2.16}) and (\ref{eq2.19}) imply that $(z-z_0)^{l-m-1}$ divides
    $\tilde{\tilde{g}}(z)$, we have

     $$l-m-1 \leq \deg(\tilde{\tilde{g}}) \leq (m'+m+1)(s+t-1)$$
    $$\Rightarrow T(l_0 +l') + tm' +m -m-1  \leq(m'+m+1)(s+t-1)= m'(s-1)+(m+1)(s+t-1)+tm'$$
     \begin{align*}
     \Rightarrow T(l_0 + l') & \leq  m'(s-1)+(m+1)(s+t)-m\\
     & \leq m'(s-1) + (m+1)(s+t)\\
     & \leq m'(s-1)+l_0(s+t)\\
     & < (m'+l_0)S + T\\
     & \leq (m'+l_0+1)T\\
     & \leq (l' + l_0)T
     \end{align*}
      which is again absurd.

      \medskip

      {\bf Case-2.2:} Finally, we suppose that $P[f]- \omega$ has no zero at all. Then $l=0$ in (\ref{eq2.17}),
      which gives
      $$P[f](z) = \omega(z) + \dfrac{B}{\prod^{t}_{j=1} {(z-\beta_j)^{p_j(l_0 + l')+m'}}},$$
    where $B\neq 0$ is a constant and so,
    $$P^{(m+1)} = B \dfrac{h(z)}{\prod^{t}_{j=1} {(z-\beta_j)^{p_j(l_0 + l')+m'+m+1}}},$$
    where $\deg(h) \leq (m+1)t-1 <(m+1)t \leq l_0 t.$\\
    Now, by proceeding as in the Case-2.1, we get a contradiction.
    \end{proof}

\section{Proofs of Main Results}

    Since  normality is a local property, we shall assume  $D$ to be the open unit disk $\mD $, throughout.\\

      \begin{proof}[\textbf {Proof of Theorem 1.3}] Suppose on the contrary that $\mathcal{F}$ is not normal at
      $z=0.$ We consider the following cases:

    \medskip

    {\bf Case-1:} Let $h(0)\neq 0.$ Then by Zalcman's Lemma(\cite{LZ}, p.216), there exist a sequence
    $\{ f_j\}$ in  $\mathcal{F},$  a sequence $\{ z_j\}$ of complex numbers in $\mD$ with $z_j \longrightarrow 0
    \text{ as } j\longrightarrow \infty,$ and a sequence $\{\rho_j\}$ of positive real numbers with $\rho_j \longrightarrow 0
    \text{ as } j\longrightarrow \infty$ such that the sequence $g_j (z): = \rho_j^{-\alpha} f_j(z_j+\rho_j z)$ converges
    locally uniformly with respect to the spherical metric to a non-constant meromorphic function $g(z)$ having
    bounded spherical derivative on $\mC$. Clearly, $(g_j ^{l_i})^{(m_i)}{\rightarrow} (g
    ^{l_i})^{(m_i)}$ and so $P[g_j] {\rightarrow} P[g]$ locally uniformly on $\mC.$\\

     Since $g$ is non-constant and $l_i \geq m_i; \forall i=1,2, \cdots, k,$ it follows that $P[g] \not\equiv 0.$ \\
		We claim that $P[g]$ is non-constant. For, suppose that 
		\begin{equation}
		P[g]\equiv a, a\in \mC \setminus \{0 \}.
		\label{eq3.1}
     \end{equation}
		Then, by definition of $P[g]$ with $l_0>0$ and $l_i \geq m_i , \forall i,$ we can see that $g$ is entire and non-vanishing. So, for  some $c\neq 0$, $g(z)=e^{cz+d} \Rightarrow P[g](z)= \prod^{k}_{i=1} (l_i c)^{m_i}e^{(l_0 +l')(cz+d)}, $ which is non-constant, a contradiction to ( \ref {eq3.1}). Hence the claim follows. Taking $\alpha
      ={m'}/{(l_0+l')}$, we find that $P[g_j](z)= P[f_j](z_j + \rho_j
      z).$ Thus on every compact subset of $\mC$, not containing poles of $g$,
      we have that
    $P[f_j](z_j + \rho_j z)- h(z_j +\rho_j z)= P[g_j](z)-h(z_j +\rho_j z) \longrightarrow P[g](z) -h(0)=P[g](z)-h_0,$
    spherically uniformly, where $h_0 = h(0) \neq 0.$

    \medskip

    In view of Theorem \ref{THM1} and Theorem \ref{THM3}, let $u_0$ and $v_0$ be two distinct zeros of $P[g]-h_0$ in $\mC$.
    Since zeros are isolated, we consider two non-intersecting neighbourhoods  $N(u_0)$ and $N(v_0)$ such that
    $N(u_0) \cup N(v_0)$ does not contain any other zero of
    $P[g]-h_0.$ By Hurwitz theorem we find that for sufficiently large values of $j$, there exist points
    $u_j \in N(u_0) \text{ and } v_j \in N(v_0)$ such that
    $$ P[f_j](z_j + \rho_j u_j)- h(z_j + \rho_j u_j)=0,$$
    and
    $$ P[f_j](z_j + \rho_j v_j)- h(z_j + \rho_j v_j)=0.$$
    Since $P[f] \text{ and } P[g]$ share $h$ IM in $\mD ,$ for each pair $f,g$ of members of $\mathcal{F},$
    for a fixed $n$ and for all $j$, we have
    $$P[f_n](z_j + \rho_j u_j)-h(z_j + \rho_j u_j)=0$$
    and
    $$P[f_n](z_j + \rho_j v_j)-h(z_j + \rho_j v_j)=0.$$
    Taking  $j \longrightarrow \infty,$ and noting that $z_j + \rho_j u_j \longrightarrow 0
    \text{ and } z_j + \rho_j v_j \longrightarrow 0,$ we find that
    $P[f_n](0)-h(0)=0.$ That is, $P[f_n](0)=h(0)=h_0 \neq 0.$ Since the zeros of $P[f_n]-h $ have  no accumulation point,
    for sufficiently large $j$, we have
    $$z_j + \rho _j u_j =0 = z_j + \rho _j v_j $$
    $$\Rightarrow u_j = -\frac{z_j}{\rho_j} = v_j$$
    $$\Rightarrow N(u_0)\cap N(v_0) \neq \phi,$$
    and this is a contradiction since the neighbourhoods $N(u_0)$ and
    $N(v_0)$ are non-intersecting.

    \medskip

    {\bf Case-2:} Suppose $h(0) = 0.$ Then, we can write $h(z) = z^m h_1(z),$ where $m \in \mN, h_1(z)$ is a
     holomorphic function in $\mD$ such that $h_1(0)\neq 0$. We may take
    $h_1(0)=1.$  Since $ 0< {m+m'}/{(l_0 +l')} <1,$ as in Case-1 by Zalcman's Lemma(\cite{LZ}, p.216),
    we obtain a sequence of rescaled functions $g_j (z)= \rho_j ^{- {(m+m')}/{(l_0 + l')}} f_j (z_j + \rho_j z)$
     that converges locally uniformly with respect to the spherical metric to a non-constant meromorphic function
     $g(z)$ on $\mC$ having bounded spherical derivatives.

     \medskip

     Further, we consider the following two subcases of Case-2:\\

     {\bf Case-2.1:} Suppose there exists a subsequence of ${z_j}/{\rho_j}$, for convenience we take ${z_j}/{\rho_j}$
      itself, such that ${z_j}/{\rho_j} \longrightarrow \infty \text{ as } j\longrightarrow \infty.$ Then consider the family
     $$\mathcal{G}: = \left\{ G_j (z)= z_j ^{-\frac{m+m'}{l_0 + l'}}f_j (z_j + z_j z) :
       f_j \in \mathcal{F} \right\}$$
    defined on $\mD ,$ for which we have
     \begin{align*}
     P[G_j](z) & = G_j ^{l_0} (G_j^{l_1})^{(m_1)} \cdots (G_j ^{l_k})^{(m_k)} (z)\\
     & = z_j ^{-\frac{m+m'}{l_0 +l'} (l_0 + l') +m'} P[f_j](z_j + z_j z)\\
      & = z_j^{-m} P[f_j](z_j + z_j z)
     \end{align*}
     That is, $P[f_j](z_j + z_j z) = z_j ^m P[G_j](z).$

     \medskip

     Now, by hypothesis, for $f_a, f_b \in \mathcal{F}$, we have
     $$(P[f_a]-h) (z_j + z_j z)=0 \Leftrightarrow (P[f_b]-h)(z_j + z_j z )=0$$
     $$\Rightarrow z_j ^m P[G_a](z) = z_j ^m (1+z)^m h_1 (z_j + z_j z) \Leftrightarrow z_j ^m P[G_b](z)=
     z_j^m (1+z)^m h_1 (z_j + z_j z) $$
     $$\Rightarrow P[G_a](z)= (1+z)^m h_1(z_j + z_j z) \Leftrightarrow P[G_b](z)=(1+z)^m h_1(z_j + z_j z).$$
     Since $(1+z)^m h_1(z_j + z_j z)\neq 0$ at the origin, it follows from Case-1 that $\mathcal{G}$ is normal in $\mD$
     and hence there exists a subsequence of $\{G_j \}$ in $\mathcal{G}$, we may take $\{G_j \}$ itself, such
     that $G_j \longrightarrow G,$ locally uniformly on $\mD$ with respect to the spherical metric.

     \medskip

     Next, if $G(0) \neq 0,$ then we see that 
     $$g_j(z)  = \rho_j ^{- \frac{m+m'}{l_0 + l'}} f_j (z_j +\rho_j z)
      = \left( \frac{z_j}{\rho_j} \right)^{\frac{m+m'}{l_0 + l'}} z_j ^{-\frac{m+m'}{l_0 + l'}} f_j (z_j + \rho_j z)\
      = \left( \frac{z_j}{\rho_j} \right)^{\frac{m+m'}{l_0 + l'}} G_j \left(\frac{\rho_j}{z_j} z \right)$$
       which converges locally uniformly with respect to the spherical metric to $\infty,$ on
      $\mC.$ This implies that $g(z)\equiv \infty,$ which is a contradiction. Thus we must have
      $G(0)=0 \Rightarrow G'(0)\neq \infty.$\\
     Next, since for each $z\in \mC$

     $$g'_j(z)  = \rho_j ^{-\frac{m+m'}{l_0 + l'} +1} f'_j(z_j + \rho_j z)
      = \left(\frac{\rho_j}{z_j} \right)^{-\frac{m+m'}{l_0 + l'} +1} G'_j
     (\frac{\rho_j}{z_j}z)$$
     and ${m+m'}/{l_0 + l'}< 1$, therefore, $g'_j(z)\longrightarrow 0$
      spherically uniformly  as $ j\longrightarrow \infty.$ This implies that $ g$ is constant, a contradiction.

      \medskip

      {\bf Case-2.2:} Suppose there exists a subsequence of ${z_j}/{\rho_j},$ for simplicity, we take ${z_j}/{\rho_j}$
      itself, such that ${z_j}/{\rho_j} \longrightarrow c \text{ as } j \longrightarrow \infty,$
      where $c$ is a finite number.  Then, we have
      $$H_j(z) = \rho_j ^{-\frac{m+m'}{l_0 +l'}}f_j (\rho_j z)= g_j \left(z-\frac{z_j}{\rho_j} \right)
      \stackrel{\chi}{\rightarrow} g(z-c) :=H(z) $$
      on $\mC.$ Since $P[H_j] (z) = \rho_j^{-m} P[f_j] (\rho_j z),$ $P[f_j] (\rho_j z) = \rho_j^{m}
      P[H_j](z)$. Also since for each $f_a$ and $f_b \text{ in } \mathcal{F}, P[f_a]$ and $P[f_b]$ share $h$ IM, it follows that
      \begin{equation}
      P[f_a](\rho_j z) = h(\rho_j z) \Leftrightarrow P[f_b] (\rho_j z) = h(\rho_j z). \label{eq2.21}
      \end{equation}
      That is,
      \begin{equation}
        P[H_a](z)= z ^m h_1 (\rho_j z) \Leftrightarrow P[H_b] (z) = z ^m h_1 (\rho_j z)
        \end{equation}
        We claim that $P[H](z)\not\equiv z^m$. In fact, if $P[H] \equiv z^m,$ then $z = 0$ is the only possible
        zero of $H.$ If $H$ is  transcendental, then $H(z) = z^{\alpha} e^{Q(z)},$ for some non-negative integer $\alpha$ and
    a polynomial $Q.$  Thus $(H^{l_i})^{(m_i)}(z) = p(z) e^{l_i Q(z)},$ where $p(z) (\not\equiv 0)$ is a rational function.
    It follows that $P[H]$ is also transcendental, which is not the case. Now if $H$ is  rational and  $z = 0$ is a zero of $H$,
    then $H$ is a polynomial . Clearly, $\deg (P[H]) \geq l_0+ 1 > m$ which is again a contradiction.

    \medskip

    On compact subsets of $\mC$, not containing poles of $H$, we
    have that $P[H_j](z) - z^m h_1 (\rho_j z)$  $ \longrightarrow P[H](z)- z^m,$ spherically uniformly.
    Since $P[H](z)\not\equiv z^m,$ by Theorem \ref{THM1} and Theorem \ref{THM3}, $P[H](z)-z^m$ has at least two distinct zeros
     in $\mC.$ Now proceeding in the same way as in Case-1, we arrive at  a contradiction.\\
    Hence in all the possible cases $\mathcal{F}$ must be normal in $\mD.$
    \end{proof}

 \bigskip

    \begin{proof}[\textbf { Proof of Theorem 1.5}] Irrespective of any of the
    conditions (i), \ (ii), and (iii), the ideas used in Case-1 of
    the proof of Theorem \ref{THM4} lead us to the conclusion
    that $P[g](z)\not\equiv h(0)=h_0 \text{ in } \mC.$

    \medskip

    If condition (i) holds, then we claim that $P[g](z)- h_0$ has at most one zero in $\mC$ which would be in
    violation to the conclusions of Theorem \ref{THM1} and Theorem \ref{THM3} thereby proving the normality of
    $\mathcal{F}$. For, suppose that $P[g](z)-h_0$ has atleast two distinct zeros, say $\zeta_0$ and $\zeta_0 ^{*}$.
     By Hurwitz theorem, there exist points $\zeta_j \longrightarrow \zeta_0 \text{ and } \zeta_j ^{*} \longrightarrow
     \zeta_0^{*}$ such that
    $$P[f_j](z_j + \rho_j \zeta_j) - h(z_j + \rho_j \zeta_j) =0$$
    and
    $$P[f_j](z_j + \rho_j \zeta_j^{*}) - h(z_j + \rho_j \zeta_j^{*}) =0,$$
    for sufficiently large $j.$\\
    Since $P[f_j](z_j + \rho_j z) -h(z_j + \rho_j z)$ has at most one zero,
    which leads to a contradiction to the fact that $\zeta_0$ and $\zeta_0^{*}$ are distinct. Hence the claim follows.

    \medskip

    Next we prove the normality of $\mathcal{F}$ when condition (ii) holds.
    By Theorem \ref{THM1} and Theorem \ref{THM3}, $P[g](z)-h_0$ must have a zero,
     say $\zeta_0$ and hence $ g(\zeta_0)\neq \infty.$ Further, by Hurwitz theorem, for sufficiently large $j,$
     there exists a sequence $\{ \zeta_j \} $ converging to $ \zeta_0 $ such that
     $$P[f_j](z_j + \rho_j \zeta_j)-h(z_j + \rho_j \zeta_j)=0$$
     Thus, by hypothesis, we have
     $$|g_j(\zeta_j)| = \rho_j^{-\frac{m'}{l_0 + l'}} |f_j (z_j + \rho_j \zeta_j)| \geq \rho_j^{- \frac{m'}{l_0 + l'}} M$$
     Since $g(\zeta_0) \neq \infty$ in some neighbourhood $N$ of $\zeta_0$, it follows that for sufficiently large values of
     $j, \ g_j(z)$ converges uniformly to $g(z)$ in $N.$ Thus for given $\epsilon >0$ and for every $z \in N,$ we have
     $$|g_j(z) - g(z)|< \epsilon$$ for sufficiently large $j.$  Therefore, for sufficiently large values of $j,$ we have
     $$|g(\zeta_j)| \geq |g_j (\zeta_j)|- |g(\zeta_j) -g_j(\zeta_j)| > \rho_j^{- \frac{m'}{l_0 + l'}}M -\epsilon$$
     which implies that $g$ has a pole at $\zeta_0,$ which is not the case. \\

    \medskip

     Finally, we prove the normality of $\mathcal{F}$ when
     condition (iii) holds. As done in the preceding discussion, we
     find that

    $$P[f_j](z_j + \rho_j \zeta_j)-h(z_j + \rho_j \zeta_j)=0.$$
  Since $\alpha = {m'}/{(l_0 + l')}$; for some positive $l_i \textit{ and } m_i,$ we have
  $$|(g_j ^{l_i})^{(m_i)}(\zeta_j)| = \rho_j ^{m_i - \alpha l_i} |(f_j ^{l_i})^{(m_i)} (z_j + \rho_j \zeta_j)|$$
  $$ \leq M \rho_j ^{m_i - \frac{m' l_i}{l_0 + l'}} \longrightarrow 0 \textit{ as } j \longrightarrow \infty.$$
  Thus,
  $$(g^{l_i})^{(m_i)} (\zeta_0) = \lim_{j \longrightarrow \infty} (g_j ^{l_i})^{(m_i)} (\zeta_j) =0$$
  $$\Rightarrow P[g] (\zeta_0) =0 \neq h_0,$$
  which is not true.
  \end{proof}

\bibliographystyle{plain}

\end{document}